\documentclass[12pt]{amsart}%
\usepackage{amsmath}
\usepackage{amsfonts}
\usepackage{amssymb}
\usepackage{graphicx,delarray}
\usepackage{graphicx}%
\providecommand{\U}[1]{\protect\rule{.1in}{.1in}}
\theoremstyle{plain}
\newtheorem{theorem}{Theorem}
\newtheorem*{theorem*}{Theorem}

\newtheorem{lemma}{Lemma}
\newtheorem{corollary}{Corollary}

\newtheorem{remark}{Remark}
\numberwithin{equation}{section}
\allowdisplaybreaks
\begin{document}

%%%\title{Some remarks on the generalized Apostol-Bernoulli and Apostol-Euler polynomials.}
%%%\author{Mohamed Amine Boutiche}
%%%\email{mboutiche@gmail.com}
%%%\author{Mourad Rahmani}
%%%\address{USTHB, Faculty of Mathematics, P. O. Box 32, El Alia,16111, Algiers, Algeria}
%%%\email{mrahmani@usthb.dz}
%%%
%%%\begin{abstract}
%%%In this note, we present several explicit formulas of the the generalized
%%%Apostol-Bernoulli and Apostol-Euler polynomials.
%%%

%%%
%%%\end{abstract}
%%%\maketitle

\title{On the higher derivatives of the inverse tangent function}

%%% AUTHOR(S) FULL NAMES, AND EMAIL ADDRESSES
\author{Mohamed Amine Boutiche}
\email{mboutiche@usthb.dz}
%%%
\author{Mourad Rahmani}
\email{mrahmani@usthb.dz}
%%% ENTER AUTHOR(S) AFFILIATION(S)
\address{Faculty of Mathematics, University of Sciences and Technology Houari Boumediene, BP 32, El Alia, Bab Ezzouar 16111, Algiers, Algeria}
%\address[affil2]{Second Author's Affiliation}
%%% AND CORRESPONDINGLY FOR OTHER AUTHORS, IF THERE ARE MORE AUTHORS
%%% ENTER ABBREVIATED AUTHOR(S) NAMES FOR PAGE HEADINGS
%\newcommand{\AuthorNames}{M. A. Boutiche, M. Rahmani}
%%% IF THERE ARE MORE THAN TWO AUTHORS WRITE
%%% \newcommand{\AuthorNames}{First Author et al.}
%%%

%%% ENTER MSC, KEYWORDS, RECEIVED, EDITOR, THANKS FOR FINANCIAL SUPPORT FOR RESEARCH

%%%%%%%%%%%%%%%%%%%%%%%%%%%%%%%%%%%%%%%%%%%%%%%%%%

\begin{abstract}
In this paper, we find explicit formulas for higher order
derivatives of the inverse tangent function. More precisely, we
study polynomials which are induced from the higher-order
derivatives of $\arctan(x)$. Successively, we give generating
functions, recurrence relations and some particular properties for
these polynomials. Connections to Chebyshev, Fibonacci, Lucas and
Matching polynomials are established. 
\end{abstract}
\maketitle

\section{Introduction}

The problem of establishing closed formulas for the $n$-derivative
of the function $\arctan(x)$ is not straightforward and has been
proved to be important for deriving rapidly convergent series for
$\pi$ \cite{AL10, AL16, Lampret11}. Recently, Many authors investigated the aforementioned problem and derive simple explicit closed-form higher derivative formulas for some classes of functions. In \cite{Boyadzhiev07, Adamchik07,	Cvijovic09} and references therein, the authors found explicit forms
of the derivative polynomials of the hyperbolic, trigonometric
tangent, cotangent and secant functions. Several new closed formulas
for higher-order derivatives have been established for trigonometric
and hyperbolic functions in \cite{XC15}, tangent and cotangent
functions in \cite{Qi15} and arc-sine functions in \cite{QZ15}. In the present work, we study polynomials which are induced from the higher-order derivatives of $\arctan(x)$.

We consider the problem of finding the $n$th derivative of $\arctan(x)$. It is easy to see that there exists a real sequence of polynomials
$$P_{n}(x)=(-1)^{n}n!\operatorname{Im}((x+i)^{n+1})$$ such that
\begin{align}
\frac{d^{n}}{dx^{n}}\left(  \arctan x\right)   &  =\frac{d^{n-1}}{dx^{n-1}%
}\left[  \frac{1}{2i}\left(  \frac{1}{x-i}-\frac{1}{x+i}\right)
\right]
\nonumber\\
&  =\frac{d^{n-1}}{dx^{n-1}}\left[  \operatorname{Im}\left(  \frac{1}%
{x-i}\right)  \right] \nonumber\\
&  =\frac{P_{n-1}\left(  x\right)  }{\left(  1+x^{2}\right)  ^{n}},
\label{pb1}%
\end{align}
where $\operatorname{Im}\left(  z\right)$, denotes the imaginary part of $z$.

\noindent By differentiation (\ref{pb1}) with respect to $x$, we get the recursion relation \cite{Lampret11}
\begin{equation}
P_{0}\left(  x\right)  =1, \ P_{n+1}\left(  x\right)  =\left(
1+x^{2}\right) P_{n}^{\prime}\left(  x\right)  -2\left(  n+1\right)
xP_{n}\left(  x\right)
. \label{rec1}%
\end{equation}
\noindent An explicit expression of $P_{n}\left(  x\right)  $ is obtained by
using the binomial formula
\begin{equation}
P_{n}\left(  x\right)  =\left(  -1\right)  ^{n}n!%
{\displaystyle\sum\limits_{k=0}^{\left\lfloor n/2\right\rfloor }}
\left(  -1\right)  ^{k}\dbinom{n+1}{2k+1}x^{n-2k}, \label{expli}%
\end{equation}
where $\left\lfloor x\right\rfloor$, denotes the integral part of
$x$, that is, the greatest integer not exceeding $x$.

\section{The fundamental properties of the alpha and beta polynomials}
We may rewrite
\begin{align}
\beta_{n}\left(  x\right)   & : =\left(  -1\right)
^{n}\frac{P_{n}\left(
	x\right)  }{n!}\nonumber\\
&  =\operatorname{Im}((x+i)^{n+1})\label{imm}\\
&  =%
{\displaystyle\sum\limits_{k=0}^{n}}
\left(  -1\right)  ^{k}\binom{n+1}{k+1}\cos\left(
\frac{k\pi}{2}\right) x^{n-k}.\nonumber\\
& =\left(  n+1\right)  x^{n}\text{ }_{2}F_{1}\left(  -\frac{n}{2},\frac{1}%
{2}-\frac{n}{2};\frac{3}{2};-\frac{1}{x^{2}}\right), \nonumber
\end{align}
where $_{2}F_{1}\left(  a,b;c;z\right)  $ denotes the Gaussian hypergeometric
function defined by%
\[
_{2}F_{1}\left(  a,b;c;z\right)  =%
%TCIMACRO{\dsum \limits_{n\geq0}}%
%BeginExpansion
{\displaystyle\sum\limits_{n\geq0}}
%EndExpansion
\frac{\left(  a\right)  _{n}\left(  b\right)  _{n}}{\left(  c\right)  _{n}%
}\frac{z^{n}}{n!},
\]
and $\left(  x\right)  _{n}$ is the Pochhammer symbol: $\left(  x\right)
_{0}=1,$ $\ \left(  x\right)  _{n}=x\left(  x+1\right)  \cdots\left(
x+n-1\right)$.

In $1755$, Euler derived the well-known formula
\[
\arctan\left(  x\right)  =%
%TCIMACRO{\dsum \limits_{n\geq0}}%
%BeginExpansion
{\displaystyle\sum\limits_{n\geq0}}
%EndExpansion
\frac{2^{2n}\left(  n!\right)  ^{2}}{\left(  2n+1\right)  !}\frac{x^{2n+1}%
}{\left(  1+x^{2}\right)  ^{n+1}}.
\]
As immediate application of (\ref{imm}), we obtain another expansion of the inverse tangent function.
\begin{theorem}
We have
\[
\arctan\left(  x\right)  =%
{\displaystyle\sum\limits_{n\geq0}}
\frac{\beta_{n}(x)}{n+1}\frac{x^{n+1}}{\left(  1+x^{2}\right)  ^{n+1}}.
\]
\end{theorem}
\begin{proof}
	From (\ref{expli}) and
	\[
	\arctan\left(  x\right)  =%
	%TCIMACRO{\dsum \limits_{n\geq1}}%
	%BeginExpansion
	{\displaystyle\sum\limits_{n\geq1}}
	%EndExpansion
	\left(  -1\right)  ^{n+1}\frac{d^{n}}{dx^{n}}\arctan\left(  x\right)
	\text{ }\frac{x^{n}}{n!},
	\]
	we get the 	desired result.
\end{proof}

Now, we give some fundamental results concerning $\beta_{n}(x)$.

\begin{theorem}
	[Generating function]The ordinary generating function of
	$\beta_{n}\left(
	x\right)  $ is given by%
	\begin{equation}
	f_{x}\left(  z\right)  =%
	%TCIMACRO{\dsum \limits_{n\geq0}}%
	%BeginExpansion
	{\displaystyle\sum\limits_{n\geq0}}
	%EndExpansion
	\beta_{n}\left(  x\right)  z^{n}=\frac{1}{1-2xz+\left(  1+x^{2}\right)  z^{2}%
	}. \label{gen}%
	\end{equation}
	
\end{theorem}

\begin{proof}
	We have%
	\begin{align*}
	f_{x}\left(  z\right)   &  =%
	%TCIMACRO{\dsum \limits_{n\geq0}}%
	%BeginExpansion
	{\displaystyle\sum\limits_{n\geq0}}
	%EndExpansion
	\left(  xz\right)  ^{n}%
	%TCIMACRO{\dsum \limits_{k\geq0}}%
	%BeginExpansion
	{\displaystyle\sum\limits_{k\geq0}}
	%EndExpansion
	\binom{n+1}{k+1}\cos\left(  \frac{k\pi}{2}\right)  \left(  -x\right)  ^{-k}\\
	&  =%
	%TCIMACRO{\dsum \limits_{n\geq0}}%
	%BeginExpansion
	{\displaystyle\sum\limits_{n\geq0}}
	%EndExpansion
	\left(  xz\right)  ^{n}\operatorname{Re}\left(
	%TCIMACRO{\dsum \limits_{k\geq0}}%
	%BeginExpansion
	{\displaystyle\sum\limits_{k\geq0}}
	%EndExpansion
	\binom{n+1}{k+1}\left(  -\frac{i}{x}\right)  ^{k}\right) \\
	&  =%
	%TCIMACRO{\dsum \limits_{n\geq0}}%
	%BeginExpansion
	{\displaystyle\sum\limits_{n\geq0}}
	%EndExpansion
	\left(  xz\right)  ^{n}\operatorname{Re}\left(
	\frac{i}{x^{n}}\left(
	x-i\right)  ^{n+1}-ix\right) \\
	&  =\frac{1}{2}%
	%TCIMACRO{\dsum \limits_{n\geq0}}%
	%BeginExpansion
	{\displaystyle\sum\limits_{n\geq0}}
	%EndExpansion
	z^{n}\left(  i\left(  x-i\right)  ^{n+1}-ix\left(  x+i\right)
	^{n+1}\right)
	\\
	&  =\frac{1}{2}i\left(  x-i\right)
	%TCIMACRO{\dsum \limits_{n\geq0}}%
	%BeginExpansion
	{\displaystyle\sum\limits_{n\geq0}}
	%EndExpansion
	\left(  z\left(  x-i\right)  \right)  ^{n}-\frac{1}{2}i\left(
	x+i\right)
	%TCIMACRO{\dsum \limits_{n\geq0}}%
	%BeginExpansion
	{\displaystyle\sum\limits_{n\geq0}}
	%EndExpansion
	\left(  z\left(  x+i\right)  \right)  ^{n}\\
	&  =\frac{1}{2}\left(  \frac{i\left(  x-i\right)  }{1-z\left(
		x-i\right) }-\frac{i\left(  x+i\right)  }{1-z\left(  x+i\right)
	}\right).
	\end{align*}
	Thus, the proof of the theorem is completed.
\end{proof}
\noindent In particular, we have
\begin{corollary} For $n\geq 0$, we have
	\[
	\beta_{n}\left(  1\right)  =2^{\frac{n+1}{2}}\cos\left(  \left(
	n-1\right) \frac{\pi}{4}\right)=2^{\frac{n+1}{2}}\sin\left(  \left(
	n+1\right) \frac{\pi}{4}\right).
	\]
\end{corollary}
\begin{proof}
	It is clear from $\operatorname{Re}\left(  i\left(  1-i\right)
	^{n+1}\right)$.
\end{proof}
\begin{theorem}
	[Generating function]The exponential generating function of
	$\beta_{n}\left(
	x\right)  $ is given by%
	\begin{equation}
		{\displaystyle\sum\limits_{n\geq0}}
	%EndExpansion
	\beta_{n}\left(  x\right) \frac{z^{n}}{n!} =(\cos(z)+x\sin(z
	))e^{xz}. \label{gen2}%
	\end{equation}
	
\end{theorem}

\begin{proof}
	From (\ref{imm}), we have%
	\begin{align*}
	\sum_{n\geq0}\operatorname{Im}\left(  \left(  x+i\right)  ^{n+1}\right)
	\frac{z^{n}}{n!}  & =\operatorname{Im}\left(  \left(  x+i\right)  \sum
	_{n\geq0}\frac{\left(  \left(  x+i\right)  z\right)  ^{n}}{n!}\right)  \\
	& =\operatorname{Im}\left(  \left(  x+i\right)  \exp\left(  \left(
	x+i\right)  z\right)  \right)  \\
	& =e^{xz}\operatorname{Im}((x+i)e^{iz})\\
	& =e^{xz}\left(  \cos z+x\sin z\right)  .
	\end{align*}
		Thus, the proof of the theorem is completed.
\end{proof}
\begin{theorem}
	[Recurrence relation]The $\beta_{n}\left(  x\right)  $ satisfy the
	following 	three-term recurrence relation
	\begin{equation}
	\beta_{n+1}\left(  x\right)  =2x\beta_{n}\left(  x\right)  -\left(
	1+x^{2}\right)  \beta_{n-1}\left(  x\right)  , \label{recc}%
	\end{equation}
	with initial conditions $\beta_{0}\left(  x\right)  =1$ and
	$\beta_{1}\left( x\right)  =2x$.
\end{theorem}

\begin{proof}
	By differentiation (\ref{gen}) with respect to $z$, we obtain%
	\[
	\left(  1-2xz+\left(  1+x^{2}\right)  z^{2}\right)
	\frac{\partial}{\partial z}f_{x}\left(  z\right)  =\left(
	2x-2\left(  1+x^{2}\right)  z\right) f_{x}\left(  z\right)  ,
	\]
	or equivalently%
	\[
	\left(  1-2xz+\left(  1+x^{2}\right)  z^{2}\right)
	{\displaystyle\sum\limits_{n\geq0}}
	n\beta_{n}\left(  x\right)  z^{n-1}=\left(  2x-2\left(
	1+x^{2}\right) z\right)
	{\displaystyle\sum\limits_{n\geq0}}
	\beta_{n}\left(  x\right)  z^{n}.
	\]
	After some rearrangement, we get%
	\[
	{\displaystyle\sum\limits_{n\geq0}}
	\left(  n+1\right)  \beta_{n+1}\left(  x\right)  z^{n}=%
	{\displaystyle\sum\limits_{n\geq0}}
	\left(  2x\left(  n+1\right)  \beta_{n}\left(  x\right)  -\left(
	1+x^{2}\right)  \left(  n+1\right)  \beta_{n-1}\left(  x\right)
	\right) z^{n}.
	\]
	Equating the coefficient of $z^{n}$, we get the result.
\end{proof}

\noindent The first few $\beta_{n}\left(  x\right)  $ are listed in Table \ref{Tab1}.
\begin{equation}%
\begin{tabular}
[c]{l}%
$\beta_{0}(x)=1$\\
$\beta_{1}(x)=2x$\\
$\beta_{2}(x)=3x^{2}-1$\\
$\beta_{3}(x)=4x^{3}-4x$\\
$\beta_{4}(x)=5x^{4}-10x^{2}+1$\\
$\beta_{5}(x)=6x^{5}-20x^{3}+6x$%
\end{tabular}
\label{Tab1}%
\end{equation}

\begin{theorem}
	The leading coefficient of $x^{n}$ in $\beta_{n}\left(  x\right)  $
	is $n+1$ and
	the following result holds true%
	\begin{equation}
	\beta_{n}\left(  -x\right)  =\left(  -1\right)  ^{n}\beta_{n}\left(
	x\right)
	. \label{par}%
	\end{equation}
	
\end{theorem}

\begin{proof}
	Hence%
	\begin{align*}
	f_{-x}\left(  -z\right)   &  =f_{x}\left(  z\right) \\%
	{\displaystyle\sum\limits_{n\geq0}}
	\beta_{n}\left(  -x\right)  \left(  -z\right)  ^{n}  &  =%
	{\displaystyle\sum\limits_{n\geq0}}
	\beta_{n}\left(  x\right)  z^{n}.
	\end{align*}
	Comparing these two series, we get (\ref{par}).
\end{proof}

\begin{remark}
	Using (\ref{par}) we can write
	\begin{equation}
	P_{n}\left(  x\right)  =n!\beta_{n}\left(  -x\right)  , \label{remarque1}%
	\end{equation}
	and the exponential generating function of $P_{n}\left(  x\right)  $
	is given
	by \[
	{\displaystyle\sum\limits_{n\geq0}}
	P_{n}\left(  x\right)  \frac{z^{n}}{n!}=\frac{1}{1+2xz+\left(
		1+x^{2}\right) z^{2}}.
	\]
	and (\ref{rec1}) becomes
	\begin{equation}
	\beta_{n+1}\left(  x\right)  =2x\beta_{n}\left(  x\right)  -\frac{1+x^{2}%
	}{n+1}\beta_{n}^{\prime}\left(  x\right)  \label{rec11}.
	\end{equation}
\end{remark}
\begin{theorem}
	For $n\geq1$, we have%
	\begin{equation}
	\frac{d}{dx}\beta_{n}\left(  x\right)  =\left(  n+1\right)
	\beta_{n-1}\left(
	x\right).    \label{dif1}%
	\end{equation}
\end{theorem}
\begin{proof}
	This follows easily from (\ref{imm}).
\end{proof}
\begin{theorem}
	[Differential Equation]$\beta_{n}\left(  x\right)  $ satisfies the
	linear 	second order ODE%
	\begin{equation}
	\left(  1+x^{2}\right)  \beta_{n}^{\prime\prime}\left(  x\right)
	-2nx\beta_{n}^{\prime}\left(  x\right)  +n\left(  n+1\right)
	\beta_{n}\left(
	x\right)  =0 \label{eqd1}%
	\end{equation}
\end{theorem}

\begin{proof}
	By differentiating (\ref{rec11}) and using (\ref{dif1}), we find
	(\ref{eqd1}).
\end{proof}

\begin{remark}
It is well known that the classical orthogonal polynomials are
characterized
by being solutions of the differential equation%
\[
A\left(  x\right)  \gamma_{n}^{\prime\prime}\left(  x\right)
+B\left( x\right)  \gamma_{n}^{\prime}\left(  x\right)
+\lambda_{n}\gamma_{n}\left( x\right)  =0,
\]
where $A$ and $B$ are independant of $n$ and $\lambda_{n}$
independant of $x$. It is obvious that the $\ \beta_{n}\left(
x\right)  $ are non-classical orthogonal polynomials.
\end{remark}

Using matrix notation, (\ref{recc}) can be written as%
\[%
\begin{pmatrix}
\beta_{r+1}\left(  x\right)  & \beta_{r+2}\left(  x\right)
\end{pmatrix}
=%
\begin{pmatrix}
\beta_{r}\left(  x\right)  & \beta_{r+1}\left(  x\right)
\end{pmatrix}%
\begin{pmatrix}
0 & -\left(  1+x^{2}\right) \\
1 & 2x
\end{pmatrix}
.
\]
Therefore%
\[%
\begin{pmatrix}
\beta_{n+r}\left(  x\right)  & \beta_{n+r+1}\left(  x\right)
\end{pmatrix}
=%
\begin{pmatrix}
\beta_{r}\left(  x\right)  & \beta_{r+1}\left(  x\right)
\end{pmatrix}%
\begin{pmatrix}
0 & -\left(  1+x^{2}\right) \\
1 & 2x
\end{pmatrix}
^{n}%
\]
for $n\geq0.$ Letting $r=0,$ we get%
\[%
\begin{pmatrix}
\beta_{n}\left(  x\right)  & \beta_{n+1}\left(  x\right)
\end{pmatrix}
=%
\begin{pmatrix}
1 & 2x
\end{pmatrix}%
\begin{pmatrix}
0 & -\left(  1+x^{2}\right) \\
1 & 2x
\end{pmatrix}
^{n}.
\]
\begin{theorem}
	We have%
	\[
	\beta_{n}\left(  x\right)  =%
	\begin{pmatrix}
	1 & 2x
	\end{pmatrix}%
	\begin{pmatrix}
	0 & -\left(  1+x^{2}\right) \\
	1 & 2x
	\end{pmatrix}
	^{n}%
	\begin{pmatrix}
	1\\
	0
	\end{pmatrix}
	.
	\]
	
\end{theorem}

\noindent It follows from the general theory of determinant \cite{VD99} that
$\beta_{n}\left(  x\right)  $ is the following $n\times n$
determinant%
\[
\beta_{n}\left(  x\right)  =%
\begin{vmatrix}
2x & -\left(  1+x^{2}\right)  & 0 & \cdots & 0\\
-1 & 2x & -\left(  1+x^{2}\right)  &  & \vdots\\
0 & -1 & \ddots & \ddots & 0\\
&  & \ddots & \ddots & -\left(  1+x^{2}\right) \\
0 & \cdots & 0 & -1 & 2x
\end{vmatrix}.
\]
In order to compute the above determinant, we recall that the
Chebyshev polynomials $U_{n}\left(  x\right)  $ of the second kind
is a polynomial of degree $n$ in $x$ defined by
\[
U_{n}\left(  x\right)  =\frac{\sin\left(  n+1\right)
	\theta}{\sin\theta }\ \text{ when }\ x=\cos\theta,
\]
and can also be written as determinant identity
\begin{equation}
U_{n}\left(  x\right)  =%
\begin{vmatrix}
2x & 1 & 0 & \cdots & 0\\
1 & 2x & 1 &  & \vdots\\
0 & 1 & \ddots & \ddots & 0\\
&  & \ddots & \ddots & 1\\
0 & \cdots & 0 & 1 & 2x
\end{vmatrix}
. \label{2nd}%
\end{equation}

The next lemma is used in the proof of Theorem \ref{THAZ}
\begin{lemma}
	\label{lemme1}For $a,b,c$ non zero, we have
\begin{equation}%
\begin{vmatrix}
b & c & 0 & \cdots & 0\\
a & b & c &  & \vdots\\
0 & a & \ddots & \ddots & 0\\
\vdots &  & \ddots & \ddots & c\\
0 & \cdots & 0 & a & b
\end{vmatrix}
=\left(  \sqrt{ac}\right)  ^{n}U_{n}\left(
\frac{b}{2\sqrt{ac}}\right)  .
\label{3rd}%
\end{equation}
\end{lemma}

\begin{proof}
	From (\ref{2nd}), we have%
	\[
	\left(  \sqrt{ac}\right)  ^{n}U_{n}\left(  \frac{b}{2\sqrt{ac}}\right)  =%
	\begin{vmatrix}
	b & \sqrt{ac} & 0 & \cdots & 0\\
	\sqrt{ac} & b & \sqrt{ac} &  & \vdots\\
	0 & \sqrt{ac} & \ddots & \ddots & 0\\
	\vdots &  & \ddots & \ddots & \sqrt{ac}\\
	0 & \cdots & 0 & \sqrt{ac} & b
	\end{vmatrix}
	.
	\]
Now, by the symmetrization process \cite{Hassan}, we get the result.
\end{proof}

\begin{theorem}\label{THAZ}
	For $n\geq1,$ we have%
	\begin{align*}
	\frac{d^{n}}{dx^{n}}\left(  \arctan\left(  x\right)  \right)   &
	=\frac{\left(  n-1\right)  !}{\left(  1+x^{2}\right)  ^{n}}\operatorname{Im}%
	((i-x)^{n})\\
	&  =\frac{\left(  n-1\right)  !}{\left(  1+x^{2}\right)  ^{\frac{n+1}{2}}%
	}U_{n-1}\left(  \frac{-x}{\sqrt{1+x^{2}}}\right)  ,
	\end{align*}
	where $U_{n}$ is the $n$th Chebyshev polynomial of the second kind.
\end{theorem}

\begin{proof}
	We apply Lemma \ref{lemme1} with $a=-1,b=2x$ and $c=-\left(
	1+x^{2}\right)  $ to obtain
	\begin{equation}
	\beta_{n}\left(  x\right)  =\left(  \sqrt{1+x^{2}}\right)
	^{n}U_{n}\left(
	\frac{x}{\sqrt{1+x^{2}}}\right)  . \label{explicit beta}%
	\end{equation}
	From (\ref{pb1}) and (\ref{remarque1}), we get the desired result.
\end{proof}

\begin{corollary}
	For $n\geq1,$ we have
	\begin{align*}
	\frac{d^{n}}{dx^{n}}\left(  \tanh^{-1}\left(  x\right)  \right)   &
	=\frac{\left(  n-1\right)  !}{2\left(  1-x^{2}\right)  ^{n}}((x+1)^{n}%
	-(x-1)^{n})\\
	&  =\frac{1}{i^{n-1}}\frac{\left(  n-1\right)  !}{\left(
		1-x^{2}\right) ^{\frac{n+1}{2}}}U_{n-1}\left(
	\frac{ix}{\sqrt{1-x^{2}}}\right)
	\end{align*}
	
\end{corollary}

\begin{proof}
	Since $\tanh^{-1}\left(  x\right)  =\frac{1}{i}\arctan(ix),$ we have%
	\begin{align*}
	\frac{d^{n}}{dx^{n}}\left(  \tanh^{-1}\left(  x\right)  \right)   &
	=\frac{(-1)^{n}}{i^{n+1}}\frac{P_{n-1}\left(  ix\right)  }{\left(
		1-x^{2}\right)  ^{n}}\\
	&  =\frac{1}{i^{n-1}}\frac{P_{n-1}\left(  -ix\right)  }{\left(  1-x^{2}%
		\right)  ^{n}}.
	\end{align*}
	Thus, the proof of the Corollary is completed.
\end{proof}

\begin{theorem}
	The roots of $\beta_{n}\left(  x\right)  $ of degree $n\geq1$ has
	$n$ simple zeros in $\mathbb{R}$ at
	\begin{equation}
	x_{k}=\cot\left(  \frac{k\pi}{n+1}\right)  ,\text{ for each
	}k=1,\ldots,n.
	\label{zeros}%
	\end{equation}
	
\end{theorem}

\begin{proof}
	Since the zeros of $U_{n}\left(  z\right)  $ are
	\[
	z_{k}=\cos\left(  \frac{k\pi}{n+1}\right)  ,k=1,\ldots,n,
	\]
	it follows from (\ref{explicit beta}) and by setting
	\[
	z_{k}=\frac{x_{k}}{\sqrt{1+x_{k}^{2}}}%
	\]
	that the zeros of $\beta_{n}\left(  x\right)  $ are given by $\left(
	\ref{zeros}\right)  .$
\end{proof}

It is well-known that for any sequence of monic polynomials $p_{n}\left(
x\right)  $ whose degrees increase by one from one member to the next, they
satisfy an\ extended recurrence relation \cite{Gautschi}
\[
p_{n+1}\left(  x\right)  =xp_{n}\left(  x\right)  -\sum_{j=0}^{n}%
\genfrac{[}{]}{0pt}{}{n}{j}%
p_{n-j}\left(  x\right)  ,
\]
and the zeros of $p_{n}\left(  x\right)  $ are the eigenvalues of the $n\times
n$ Hessenberg matrix of the coefficients $%
\genfrac{[}{]}{0pt}{}{n}{j}%
%EndExpansion
$ arranged upward in the $k$th column
\[
H_{n}=%
\begin{pmatrix}%
%TCIMACRO{\QATOPD{[}{]}{0}{0}}%
%BeginExpansion
\genfrac{[}{]}{0pt}{}{0}{0}%
%EndExpansion
&
%TCIMACRO{\QATOPD{[}{]}{1}{1}}%
%BeginExpansion
\genfrac{[}{]}{0pt}{}{1}{1}%
%EndExpansion
&
%TCIMACRO{\QATOPD{[}{]}{2}{2}}%
%BeginExpansion
\genfrac{[}{]}{0pt}{}{2}{2}%
%EndExpansion
& \cdots &
%TCIMACRO{\QATOPD{[}{]}{n-2}{n-2}}%
%BeginExpansion
\genfrac{[}{]}{0pt}{}{n-2}{n-2}%
%EndExpansion
&
%TCIMACRO{\QATOPD{[}{]}{n-1}{n-1}}%
%BeginExpansion
\genfrac{[}{]}{0pt}{}{n-1}{n-1}%
%EndExpansion
\\
1 &
%TCIMACRO{\QATOPD{[}{]}{1}{0}}%
%BeginExpansion
\genfrac{[}{]}{0pt}{}{1}{0}%
%EndExpansion
&
%TCIMACRO{\QATOPD{[}{]}{2}{1}}%
%BeginExpansion
\genfrac{[}{]}{0pt}{}{2}{1}%
%EndExpansion
& \cdots &
%TCIMACRO{\QATOPD{[}{]}{n-2}{n-3}}%
%BeginExpansion
\genfrac{[}{]}{0pt}{}{n-2}{n-3}%
%EndExpansion
&
%TCIMACRO{\QATOPD{[}{]}{n-1}{n-2}}%
%BeginExpansion
\genfrac{[}{]}{0pt}{}{n-1}{n-2}%
%EndExpansion
\\
0 & 1 &
%TCIMACRO{\QATOPD{[}{]}{2}{0}}%
%BeginExpansion
\genfrac{[}{]}{0pt}{}{2}{0}%
%EndExpansion
& \cdots &
%TCIMACRO{\QATOPD{[}{]}{n-2}{n-4}}%
%BeginExpansion
\genfrac{[}{]}{0pt}{}{n-2}{n-4}%
%EndExpansion
&
%TCIMACRO{\QATOPD{[}{]}{n-1}{n-3}}%
%BeginExpansion
\genfrac{[}{]}{0pt}{}{n-1}{n-3}%
%EndExpansion
\\
\vdots & \vdots &  &  & \vdots & \vdots\\
0 & 0 & 0 & \cdots &
%TCIMACRO{\QATOPD{[}{]}{n-2}{0}}%
%BeginExpansion
\genfrac{[}{]}{0pt}{}{n-2}{0}%
%EndExpansion
&
%TCIMACRO{\QATOPD{[}{]}{n-1}{1}}%
%BeginExpansion
\genfrac{[}{]}{0pt}{}{n-1}{1}%
%EndExpansion
\\
0 & 0 & 0 & \cdots & 1 &
%TCIMACRO{\QATOPD{[}{]}{n-1}{0}}%
%BeginExpansion
\genfrac{[}{]}{0pt}{}{n-1}{0}%
%EndExpansion
\end{pmatrix}
.
\]

Let
\begin{equation}
\pi_{n}\left(  x\right)  :=\frac{\beta_{n}\left(  x\right)  }{n+1}%
,\label{BEN1}%
\end{equation}
be the monic polynomial of degree $n$.

\begin{theorem}
	For $n\geq0,$ we have%
	\begin{equation}
	\pi_{0}\left(  x\right)  =1;\text{ }\pi_{n+1}\left(  x\right)  =x\pi
	_{n}\left(  x\right)  -\sum_{j=1}^{n}\frac{2^{j+1}}{j+1}\binom{n}{j}\left\vert
	B_{j+1}\right\vert \pi_{n-j}\left(  x\right)  .\label{ZE}%
	\end{equation}
	where $B_{n}$ denote the Bernoulli numbers.
\end{theorem}

\begin{proof}
	By using generating function techniques, we can verify (\ref{ZE}) directly.
	From (\ref{BEN1}) and (\ref{gen2}), we have%
	\begin{multline*}%
%TCIMACRO{\dsum \limits_{n\geq0}}%
%BeginExpansion
{\displaystyle\sum\limits_{n\geq0}}
%EndExpansion
\left(  x\pi_{n}\left(  x\right)  -\sum_{j=1}^{n}\frac{2^{j+1}}{j+1}\binom
{n}{j}\left\vert B_{j+1}\right\vert \pi_{n-j}\left(  x\right)  \right)
\frac{z^{n}}{n!}\\
=x%
%TCIMACRO{\dsum \limits_{n\geq0}}%
%BeginExpansion
{\displaystyle\sum\limits_{n\geq0}}
%EndExpansion
\pi_{n}\left(  x\right)  \frac{z^{n}}{n!}-\sum_{j\geq1}\frac{2^{j+1}}{\left(
j+1\right)  !}\left\vert B_{j+1}\right\vert
%TCIMACRO{\dsum \limits_{n\geq0}}%
%BeginExpansion
{\displaystyle\sum\limits_{n\geq0}}
%EndExpansion
\pi_{n-j}\left(  x\right)  \frac{z^{n}}{\left(  n-j\right)  !}\\
=\frac{1}{z}\left(  x-\frac{1}{z}\sum_{j\geq2}\frac{2^{j}}{j!}\left\vert
B_{j}\right\vert z^{j}\right)
%TCIMACRO{\dsum \limits_{n\geq1}}%
%BeginExpansion
{\displaystyle\sum\limits_{n\geq1}}
%EndExpansion
\beta_{n-1}\left(  x\right)  \frac{z^{n}}{n!}.
\end{multline*}
	Since,%
	\[
	\cot\left(  z\right)  -\frac{1}{z}=-\sum_{j\geq2}\frac{2^{j}}{j!}\left\vert
	B_{j}\right\vert z^{j-1},
	\]
	and
	\begin{align*}%
	%TCIMACRO{\dsum \limits_{n\geq1}}%
	%BeginExpansion
	{\displaystyle\sum\limits_{n\geq1}}
	%EndExpansion
	\beta_{n-1}\left(  x\right)  \frac{z^{n}}{n!}  & =%
	%TCIMACRO{\dint }%
	%BeginExpansion
	{\displaystyle\int}
	%EndExpansion
	e^{xz}\left(  \cos z+x\sin z\right)  dz\\
	& =e^{xz}\sin z.
	\end{align*}
	We get%
	\[%
	%TCIMACRO{\dsum \limits_{n\geq0}}%
	%BeginExpansion
	{\displaystyle\sum\limits_{n\geq0}}
	%EndExpansion
	\left(  x\pi_{n}\left(  x\right)  -\sum_{j=1}^{n}\frac{2^{j+1}}{j+1}\binom
	{n}{j}\left\vert B_{j+1}\right\vert \pi_{n-j}\left(  x\right)  \right)
	\frac{z^{n}}{n!}=\frac{e^{xz}}{z^{2}}\left(  \left(  xz-1\right)  \sin z+z\cos
	z\right)  .
	\]
	On the other hand, we have%
	\begin{align*}%
	%TCIMACRO{\dsum \limits_{n\geq0}}%
	%BeginExpansion
	{\displaystyle\sum\limits_{n\geq0}}
	%EndExpansion
	\pi_{n+1}\left(  x\right)  \frac{z^{n}}{n!}  & =%
	%TCIMACRO{\dsum \limits_{n\geq0}}%
	%BeginExpansion
	{\displaystyle\sum\limits_{n\geq0}}
	%EndExpansion
	\frac{\beta_{n+1}\left(  x\right)  }{n+2}\frac{z^{n}}{n!}\\
	& =%
	%TCIMACRO{\dsum \limits_{n\geq0}}%
	%BeginExpansion
	{\displaystyle\sum\limits_{n\geq0}}
	%EndExpansion
	\left(  n+1\right)  \beta_{n+1}\left(  x\right)  \frac{z^{n}}{\left(
		n+2\right)  !}\\
	& =%
	%TCIMACRO{\dsum \limits_{n\geq1}}%
	%BeginExpansion
	{\displaystyle\sum\limits_{n\geq1}}
	%EndExpansion
	\left(  n-1\right)  \beta_{n-1}\left(  x\right)  \frac{z^{n-2}}{n!}\\
	& =\frac{1}{z}%
	%TCIMACRO{\dsum \limits_{n\geq0}}%
	%BeginExpansion
	{\displaystyle\sum\limits_{n\geq0}}
	%EndExpansion
	\beta_{n}\left(  x\right)  \frac{z^{n}}{n!}-\frac{1}{z^{2}}%
	%TCIMACRO{\dsum \limits_{n\geq1}}%
	%BeginExpansion
	{\displaystyle\sum\limits_{n\geq1}}
	%EndExpansion
	\beta_{n-1}\left(  x\right)  \frac{z^{n}}{n!}\\
	& =\frac{1}{z}e^{xz}\left(  \cos z+x\sin z\right)  -\frac{1}{z^{2}}e^{xz}\sin
	z\\
	& =\frac{1}{z^{2}}e^{xz}\left(  z\cos z+\left(  zx-1\right)  \sin z\right)  .
	\end{align*}
	The theorem is verified.
\end{proof}

Now, using the fact that $B_{2n+1}=0$ for $n>1,$ we can write
\[%
%TCIMACRO{\QATOPD{[}{]}{n}{2j}}%
%BeginExpansion
\genfrac{[}{]}{0pt}{}{n}{2j}%
%EndExpansion
=0\text{ and }%
%TCIMACRO{\QATOPD{[}{]}{n}{2j+1}}%
%BeginExpansion
\genfrac{[}{]}{0pt}{}{n}{2j+1}%
%EndExpansion
=\frac{2^{2j+2}}{2j+2}\binom{n}{2j+1}\left\vert B_{2j+1}\right\vert .
\]
Then, the $n\times n$ Hessenberg matrix $H_{n}$ takes the form%
\[
H_{n}=%
\begin{pmatrix}
0 & \frac{1}{3} & 0 & \frac{2}{15} & 0 & \frac{16}{63} & \cdots & \frac{2^{n}%
}{n}\left\vert B_{n}\right\vert \\
1 & 0 & \frac{2}{3} & 0 & \frac{8}{15} & 0 & \cdots & 2^{n-1}\left\vert
B_{n-1}\right\vert \\
0 & 1 & 0 & 1 & 0 & \frac{32}{21} & \cdots & \left(  n-1\right)
2^{n-3}\left\vert B_{n-2}\right\vert \\
0 & 0 & 1 & 0 & \frac{4}{3} & 0 & \cdots & \left(  n-1\right)  \left(
n-2\right)  \frac{2^{n-4}}{3}\left\vert B_{n-3}\right\vert \\
0 & 0 & 0 & 1 & 0 & \frac{5}{3} & \cdots & \left(  n-1\right)  \left(
n-2\right)  \left(  n-3\right)  \frac{2^{n-7}}{3}\left\vert B_{n-4}\right\vert
\\
\vdots & \vdots & \vdots &  & \ddots & \ddots & \ddots & \vdots\\
0 & 0 & 0 & 0 & 0 & \cdots & 0 & \frac{1}{3}\left(  n-1\right)  \\
0 & 0 & 0 & 0 & 0 & \cdots & 1 & 0
\end{pmatrix}
,
\]
which the eigenvalues are $\lambda_{k}=\cot\left(  \frac{k\pi}{n+1}\right)  ,$
for $k=1,\ldots,n$.

It is convenient to define a companion sequence $\alpha_{n}\left(
x\right)  $ of $\beta_{n}\left(  x\right)  $ by
\begin{align}
\alpha_{n}\left(  x\right)   &  =\operatorname{Re}(\left(  x+i\right)  ^{n})\nonumber\\
&  =%
{\displaystyle\sum\limits_{k=0}^{\left\lfloor n/2\right\rfloor }}
\left(  -1\right)  ^{k}\dbinom{n}{2k}x^{n-2k} \label{ani}\\
&  =%
{\displaystyle\sum\limits_{k=0}^{n}}
%EndExpansion
\left(  -1\right)  ^{k}\binom{n}{k}\cos\left(  \frac{k\pi}{2}\right)
x^{n-k},\nonumber \\
& =x^{n}\text{ }_{2}F_{1}\left(  -\frac{n}{2},\frac{1}%
{2}-\frac{n}{2};\frac{1}{2};-\frac{1}{x^{2}}\right), \nonumber
\end{align}
where $\operatorname{Re}\left(  z\right)$, denotes the real part
of $z$. By direct computation from (\ref{ani}), we find
\[%
\begin{tabular}
[c]{l}%
$\alpha_{0}(x)=1$\\
$\alpha_{1}(x)=x$\\
$\alpha_{2}(x)=x^{2}-1$\\
$\alpha_{3}(x)=x^{3}-3x$\\
$\alpha_{4}(x)=x^{4}-6x^{2}+1$\\
$\alpha_{5}(x)=x^{5}-10x^{3}+5x$%
\end{tabular}
\]

Similarly, we obtain

\begin{theorem}\label{alpha}
	\begin{enumerate}
		\item[\ ]
		\item The ordinary generating function of $\alpha_{n}\left(  x\right)  $ is
		given by%
		\begin{equation}%
				{\displaystyle\sum\limits_{n\geq0}}
			\alpha_{n}\left(  x\right)  z^{n}=\frac{1-xz}{1-2xz+\left(
			1+x^{2}\right) z^{2}}.
		\end{equation}
		
		\item The exponential generating function of $\alpha_{n}\left(  x\right)  $ is given by%
		\begin{equation}%
		{\displaystyle\sum\limits_{n\geq0}}
		\alpha_{n}\left(  x\right) \frac{z^{n}}{n!}=\cos(z)e^{xz}.
		\end{equation}
		\item The $\alpha_{n}\left(  x\right)$ satisfy the following three-term recurrence relation%
		\[
		\alpha_{n+1}\left(  x\right)  =2x\alpha_{n}\left(  x\right)  -\left(
		1+x^{2}\right)  \alpha_{n-1}\left(  x\right)  ,
		\]
		with initial conditions $\alpha_{0}\left(  x\right)  =1$ and $\alpha
		_{1}\left(  x\right)  =x.$
		
		\item We have
		\begin{align}
		\alpha_{n}\left(  x\right)   &  =%
		\begin{pmatrix}
		1 & x
		\end{pmatrix}%
		\begin{pmatrix}
		0 & -\left(  1+x^{2}\right) \\
		1 & 2x
		\end{pmatrix}
		^{n}%
		\begin{pmatrix}
		1\\
		0
		\end{pmatrix}
		\\
		&  =%
		\begin{vmatrix}
		x & -\left(  1+x^{2}\right)  & 0 & \cdots & 0\\
		-1 & 2x & -\left(  1+x^{2}\right)  &  & \vdots\\
		0 & -1 & \ddots & \ddots & 0\\
		&  & \ddots & \ddots & -\left(  1+x^{2}\right) \\
		0 & \cdots & 0 & -1 & 2x
		\end{vmatrix}
		\\
		&  =\left(  \sqrt{1+x^{2}}\right)  ^{n}T_{n}\left(  \frac{x}{\sqrt{1+x^{2}}%
		}\right)\label{25}
		\end{align}
		where $T_{n}$ is the $n$th Chebyshev polynomial of the first kind
		defined by
		\[
		T_{n}\left(  x\right)  =\cos(n\theta)\ \text{ when }\ x=\cos\theta.
		\]
			
		\item The following result holds true%
		\begin{equation}
		\alpha_{n}\left(  -x\right)  =\left(  -1\right)
		^{n}\alpha_{n}\left( x\right)  .
		\end{equation}

		\item We have%
		\begin{equation}
		\frac{d}{dx}\alpha_{n}\left(  x\right)  =n\alpha_{n-1}\left(
		x\right)  .
		\end{equation}

		\item $\alpha_{n}\left(  x\right)  $ satisfies the linear second order ODE%
		\begin{equation}
		\left(  1+x^{2}\right)  \alpha_{n}^{\prime\prime}\left(  x\right)
		-2(n-1)x\alpha_{n}^{\prime}\left(  x\right)  +n\left(  n-1\right)
		\alpha _{n}\left(  x\right)  =0
		\end{equation}

		\item The roots of $\ \alpha_{n}\left(  x\right)  $ of degree $n\geq1$ has $n
		$ simple zeros in $\mathbb{R}$ at
		\begin{equation}
		x_{k}=\cot\left(  \frac{\left(  2k-1\right)  \pi}{2n}\right)
		,\text{ for each }k=1,\ldots,n.
		\end{equation}
		\item For $n\geq0,$ we have%
		\begin{equation}
		\alpha_{0}\left(  x\right)  =1;\text{ }\alpha_{n+1}\left(  x\right)  =x\alpha_{n}\left(  x\right)  -\sum_{j=1}^{n}\frac{2^{j+1}(2^{j+1}-1)}{j+1}\binom{n}{j}\left\vert	B_{j+1}\right\vert \alpha_{n-j}\left(  x\right)  .
		\end{equation}
		
	\end{enumerate}
\end{theorem}

\begin{theorem}
	For all $n\geq1,$ we have%
	\begin{align*}
	\alpha_{n}\left(  x\right)   &  =\beta_{n}\left(  x\right)  -x\beta
	_{n-1}\left(  x\right) \\
	\beta_{n}\left(  x\right)   &  =x\left(  1+x^{2}\right)
	\alpha_{n-1}\left( x\right)  -\left(  x^{2}-1\right)
	\alpha_{n}\left(  x\right)  .
	\end{align*}
	
\end{theorem}

\begin{proof}
	Since
	\begin{equation}
	\alpha_{n}\left(  x\right)  =\frac{\left(  x+i\right)  ^{n}+\left(
		x-i\right)  ^{n}}{2} \label{alphaa}%
	\end{equation}
	and
	\begin{equation}
	\beta_{n}\left(  x\right)  =\frac{\left(  x+i\right)  ^{n+1}-\left(
		x-i\right)  ^{n+1}}{2i}, \label{betaa}%
	\end{equation}

	we get the desired result.
\end{proof}

In the same manner, we can prove the Tur\'{a}n's inequalities for
$\alpha _{n}\left(  x\right)  $ and $\beta_{n}\left(  x\right)  .$

\begin{theorem}
	Tur\'{a}n's inequalities for $\alpha_{n}\left(  x\right)  $ and
	$\beta
	_{n}\left(  x\right)  $ are%
	\begin{align*}
	\alpha_{n}^{2}\left(  x\right)  -\alpha_{n-1}\left(  x\right)
	\alpha _{n+1}\left(  x\right)   &  =\left(  x^{2}+1\right)
	^{n-1}>0,\text{ \ for
	}n\geq1\\
	\beta_{n}^{2}\left(  x\right)  -\beta_{n-1}\left(  x\right)  \beta
	_{n+1}\left(  x\right)   &  =\left(  x^{2}+1\right)  ^{n}>0,\text{ \
		for }n\geq0.
	\end{align*}
	
\end{theorem}

\section{Connection with other sequences}

It is well known that $\tan\left(  n\arctan\left(  x\right)  \right)
$ is a rational function and is equal to the following identity \cite{Beeler}
\[
\tan\left(  n\arctan\left(  x\right)  \right)
=\frac{1}{i}\frac{\left( 1+ix\right)  ^{n}-\left(  1-ix\right)
	^{n}}{\left(  1+ix\right)  ^{n}+\left(
	1-ix\right)  ^{n}}.
\]

It follows from (\ref{alphaa}) and (\ref{betaa}) that for all
$n\geq1,$ we
have%
\begin{align*}
\tan\left(  n\arctan\left(  x\right)  \right)   &  =\left\{
\begin{array}
[c]{c}%
-\dfrac{\beta_{n-1}\left(  x\right)  }{\alpha_{n}\left(  x\right)
},\text{
	\ \ \ \ }n\text{ even}\\
\dfrac{\alpha_{n}\left(  x\right)  }{\beta_{n-1}\left(  x\right)
},\text{
	\ \ \ \ \ \ \ \ }n\text{ odd}%
\end{array}
\right. \\
&  =\left\{
\begin{array}
[c]{c}%
x-\left(  1+x^{2}\right)  \dfrac{\alpha_{n-1}\left(  x\right)
}{\alpha
	_{n}\left(  x\right)  },\text{ \ \ \ \ }n\text{ even}\\
\dfrac{\beta_{n}\left(  x\right)  }{\beta_{n-1}\left(  x\right)
}-x,\text{
	\ \ \ \ \ \ \ \ }n\text{ odd}%
\end{array}
\right.  .
\end{align*}

\subsection{Fibonacci polynomial}

Let $h\left(  x\right)  $ be a polynomial with real coefficients.
The link between Fibonacci polynomials and Chebyshev polynomials of
the second kind is
given by%
\[
F_{n,h}\left(  x\right)  =i^{n-1}U_{n-1}\left(  \frac{h\left(  x\right)  }%
{2i}\right)  ,
\]
now using (\ref{explicit beta}) we get%
\begin{align}
F_{n,h}\left(  x\right)   &  =\left(  \frac{i}{2}\right)
^{n-1}\left( \sqrt{h^{2}\left(  x\right)  +4}\right)
^{n-1}\beta_{n-1}\left( \frac{-ih\left(  x\right)
}{\sqrt{h^{2}\left(  x\right)  +4}}\right)
\nonumber\\
&  =\frac{1}{2^{n-1}}%
%TCIMACRO{\dsum \limits_{k=0}^{\left\lfloor n/2\right\rfloor }}%
%BeginExpansion
{\displaystyle\sum\limits_{k=0}^{\left\lfloor n/2\right\rfloor }}
%EndExpansion
\dbinom{n+1}{2k+1}h^{n-2k}\left(  x\right)  \left(  h^{2}\left(
x\right)
+4\right)  ^{k} \label{expfib}%
\end{align}

\subsection{Lucas polynomial}

In the same manner, Lucas polynomials and Chebyshev polynomials of
the first
kind are related by%
\[
L_{n,h}\left(  x\right)  =2i^{n}T_{n}\left(  \frac{h\left(  x\right)  }%
{2i}\right) ,
\]
Using (\ref{25}), we get
\begin{align}
L_{n,h}\left(  x\right)   &  =\frac{i^{n}}{2^{n-1}}\left(
\sqrt{h^{2}\left(
	x\right)  +4}\right)  ^{n}\alpha_{n}\left(  \frac{-ih\left(  x\right)  }%
{\sqrt{h^{2}\left(  x\right)  +4}}\right) \nonumber\\
&  =\frac{1}{2^{n-1}}%
%TCIMACRO{\dsum \limits_{k=0}^{\left\lfloor n/2\right\rfloor }}%
%BeginExpansion
{\displaystyle\sum\limits_{k=0}^{\left\lfloor n/2\right\rfloor }}
%EndExpansion
\dbinom{n}{2k}h^{n-2k}\left(  x\right)  \left(  h^{2}\left(
x\right)
+4\right)  ^{k} \label{expLuc}%
\end{align}

Note that the above formulas (\ref{expfib}) and (\ref{expLuc}) are
given in \cite{NH09} which generalize the Catalan formulas for
Fibonacci and Lucas numbers (see Koshy \cite{Koshy01} page 162).

\subsection{Matching polynomial}

The matching polynomial \cite{Far79} is a well-known polynomial in
graph theory and is defined by
\[
M_{G}(x)=\displaystyle\sum_{k=0}^{\left\lfloor n/2\right\rfloor}
(-1)^{k} m(G,k)x^{n-2k}.
\]

We know from Hosoya in \cite{Hosoya94}, a transformation of a
matching polynomial into typical orthogonal polynomials by

\begin{align*}
M_{P_{n}}(x)  &  =U_{n}(x/2),\\
M_{C_{n}}(x)  &  =2T_{n}(x/2),
\end{align*}

where $P_{n}$ and $C_{n}$ are the path and the cycle graph
respectively.

Now, by using (\ref{explicit beta}) and (\ref{25}) with an appropriate change of variables, we get

\begin{align}
M_{P_{n}}(x) &= \frac{1}{2^{n}}\sum_{k=0}^{\left\lfloor n/2\right\rfloor } (-1)^{k}
\binom{n+1}{2k+1}x^{n-2k}\left(
4-x^{2}\right)^{k} \label{expmp},\\
M_{C_{n}}(x) &= \frac{1}{2^{n-1}}\sum_{k=0}^{\left\lfloor n/2\right\rfloor
}(-1)^{k}\binom{n}{2K} x^{n-2k}\left(4-x^{2}\right)^{k}.
\label{expmc}
\end{align}
\section{Conclusion}
In our present investigation, we have studied polynomials which are
induced from the higher-order derivatives of $\arctan(x)$. We have
derived some explicit formula for higher order derivatives of the
inverse tangent function, generating functions, recurrence relations
and some particular properties for these polynomials. As a
consequence, we have established connections to Chebyshev,
Fibonacci, Lucas and Matching polynomials. We did not examine the
orthogonality of $\alpha_{n}(x)$ and $\beta_{n}(x)$ polynomials. We
believe that these polynomials are a nice example for Sobolev
orthogonal polynomials.

\section*{Acknowledgement}
The second author wish to thank Professor Gradimir V.
Milovanovi\'{c} for the fruitful discussions during his visit to the
University of Sciences and Technology (USTHB) at Algiers, Algeria.

%%% ENTER REFERENCES IN THE FORM

\end{document}